\documentclass[preprint]{article}
\pdfoutput=1

\usepackage[utf8]{inputenc}
\usepackage[T1]{fontenc}
\usepackage[english]{babel}
\usepackage{fullpage}
\usepackage{authblk}

\usepackage[numbers]{natbib}
\usepackage{graphicx} %
\usepackage{amsmath, amsthm, amssymb}
\usepackage{enumerate}
\usepackage{thmtools}
\usepackage{mathrsfs}
\usepackage{thm-restate}
\usepackage{soul}
\usepackage[usenames,dvipsnames]{xcolor} %
\usepackage{tikz}
\usepackage[labelsep=period,labelfont=bf,justification=centering]{caption}
\usepackage{float,graphicx,subcaption}
\usepackage{fullpage}
\usepackage{thm-restate}
\usepackage{hyperref}
\usepackage{comment}

\declaretheorem[name=Theorem, numberwithin=section]{theorem}
\declaretheorem[name=Lemma, sibling=theorem]{lemma}

\declaretheorem[name=Corollary, sibling=theorem]{corollary}
\declaretheorem[name=Conjecture, sibling=theorem]{conjecture}

\def\cqedsymbol{\ifmmode$\lrcorner$\else{\unskip\nobreak\hfil
\penalty50\hskip1em\null\nobreak\hfil$\lrcorner$
\parfillskip=0pt\finalhyphendemerits=0\endgraf}\fi}

\newtheoremstyle{named}{}{}{\itshape}{}{\bfseries}{.}{.5em}{\thmnote{#3's }#1}
\theoremstyle{named}

\def\dist{\text{dist}}
\def\degen{\text{degen}}
\def\diam{\text{diam}}

\newcounter{regle}

\usetikzlibrary{intersections}
\usetikzlibrary{arrows}
\usetikzlibrary{arrows.meta}
\usetikzlibrary{shapes.arrows}

\usetikzlibrary{patterns}
\usetikzlibrary{patterns.meta}
\usetikzlibrary{calc}
\usetikzlibrary{fit}
\usetikzlibrary{backgrounds, trees, hobby}

\let\le\leqslant
\let\ge\geqslant
\let\leq\leqslant
\let\geq\geqslant

\def\centerarc[#1](#2)(#3:#4:#5)%
    { \draw[#1] ($(#2)+({#5*cos(#3)},{#5*sin(#3)})$) arc (#3:#4:#5) }
\def\nodearc(#1)(#2)(#3:#4)[#5](#6)%
    { \node (#1) at ($(#2)+({#4*cos(#3)},{#4*sin(#3)})$) [#5] {#6}}
\def\nodeellipse(#1)(#2)(#3:#4:#5)[#6](#7)%
    { \node (#1) at ($(#2)+({#4*cos(#3)},{#5*sin(#3)})$) [#6] {#7}}

\def\nodc(#1)(#2,#3)[#4, #5]%
    {\node (o1) at (#2,#3) [circle, minimum size=3pt, scale = 0.6, fill=#4] {}; 
     \node (#1) at ($(#2,#3)+(0.005,-0.005)$) [circle, minimum size=12pt, line width=1.8, draw=#5] {} }

\pgfdeclarepattern{
  name=hatch,
  parameters={\hatchsize,\hatchangle,\hatchlinewidth},
  bottom left={\pgfpoint{-.1pt}{-.1pt}},
  top right={\pgfpoint{\hatchsize+.1pt}{\hatchsize+.1pt}},
  tile size={\pgfpoint{\hatchsize}{\hatchsize}},
  tile transformation={\pgftransformrotate{\hatchangle}},
  code={
    \pgfsetlinewidth{\hatchlinewidth}
    \pgfpathmoveto{\pgfpoint{-.1pt}{-.1pt}}
    \pgfpathlineto{\pgfpoint{\hatchsize+.1pt}{\hatchsize+.1pt}}
    \pgfusepath{stroke}
  }
}

\tikzset{
  hatch size/.store in=\hatchsize,
  hatch angle/.store in=\hatchangle,
  hatch line width/.store in=\hatchlinewidth,
  hatch size=5pt,
  hatch angle=0pt,
  hatch line width=.5pt,
}

\tikzset{
    dashone/.style={dash pattern=on 5pt off 15pt},
  }
  
\pgfdeclarelayer{background}
\pgfsetlayers{background,main}

\newcommand{\convexpath}[2]{
[   
    create hullnodes/.code={
        \global\edef\namelist{#1}
        \foreach [count=\counter] \nodename in \namelist {
            \global\edef\numberofnodes{\counter}
            \node at (\nodename) [draw=none,name=hullnode\counter] {};
        }
        \node at (hullnode\numberofnodes) [name=hullnode0,draw=none] {};
        \pgfmathtruncatemacro\lastnumber{\numberofnodes+1}
        \node at (hullnode1) [name=hullnode\lastnumber,draw=none] {};
    },
    create hullnodes
]
($(hullnode1)!#2!-90:(hullnode0)$)
\foreach [
    evaluate=\currentnode as \previousnode using \currentnode-1,
    evaluate=\currentnode as \nextnode using \currentnode+1
    ] \currentnode in {1,...,\numberofnodes} {
  let
    \p1 = ($(hullnode\currentnode)!#2!-90:(hullnode\previousnode)$),
    \p2 = ($(hullnode\currentnode)!#2!90:(hullnode\nextnode)$),
    \p3 = ($(\p1) - (hullnode\currentnode)$),
    \n1 = {atan2(\y3,\x3)},
    \p4 = ($(\p2) - (hullnode\currentnode)$),
    \n2 = {atan2(\y4,\x4)},
    \n{delta} = {-Mod(\n1-\n2,360)}
  in 
    {-- (\p1) arc[start angle=\n1, delta angle=\n{delta}, radius=#2] -- (\p2)}
}
-- cycle
}

\newcommand{\hobbyconvexpath}[2]{
[   
    create hobbyhullnodes/.code={
        \global\edef\namelist{#1}
        \foreach [count=\counter] \nodename in \namelist {
            \global\edef\numberofnodes{\counter}
            \node at (\nodename)
[draw=none,name=hobbyhullnode\counter] {};
        }
        \node at (hobbyhullnode\numberofnodes)
[name=hobbyhullnode0,draw=none] {};
        \pgfmathtruncatemacro\lastnumber{\numberofnodes+1}
        \node at (hobbyhullnode1)
[name=hobbyhullnode\lastnumber,draw=none] {};
    },
    create hobbyhullnodes
]
($(hobbyhullnode1)!#2!-90:(hobbyhullnode0)$)
\pgfextra{
  \gdef\hullpath{}
\foreach [
    evaluate=\currentnode as \previousnode using int(\currentnode-1),
    evaluate=\currentnode as \nextnode using int(\currentnode+1)
    ] \currentnode in {1,...,\numberofnodes} {
    \ifnum\currentnode=1\relax
    \xdef\hullpath{([closed=true]$(hobbyhullnode\currentnode)!#2!180:(hobbyhullnode\previousnode)$)
  ..($(hobbyhullnode\nextnode)!0.5!(hobbyhullnode\currentnode)$)}
    \else
    \xdef\hullpath{\hullpath
  ..($(hobbyhullnode\currentnode)!#2!180:(hobbyhullnode\previousnode)$)
  ..($(hobbyhullnode\nextnode)!0.5!(hobbyhullnode\currentnode)$)}
    \fi
    \ifx\currentnode\numberofnodes
    \else
    \xdef\hullpath{\hullpath
  ..($(hobbyhullnode\nextnode)!#2!-90:(hobbyhullnode\currentnode)$)}
    \fi
}
}
\hullpath
}

\newcommand{\colora}{Gray}
\newcommand{\colorb}{SkyBlue}
\newcommand{\colorc}{RawSienna}
\newcommand{\colord}{YellowOrange}
\newcommand{\colore}{WildStrawberry}
\newcommand{\colorf}{ForestGreen}
\newcommand{\colorg}{Violet}

\title{Optimal List Recoloring of Subcubic Graphs and Complete Multipartite Graphs }
\author[1]{Lucas De Meyer}

\affil[1]{Université Claude Bernard Lyon 1, CNRS, Ecole Centrale de Lyon, INSA Lyon, LIRIS, UMR5205, 69622 Villeurbanne, France}

\begin{document}

\maketitle

\begin{abstract}
    For a list-assignment $L$, the \emph{reconfiguration graph} $C_L(G)$ of a graph $G$ is the graph whose vertices are proper $L$-colorings of $G$ and whose edges link two colorings that differ on only one vertex. 
    If $|L(v)| \ge d(v) + 2$ for every vertex of $G$, it is known that $C_L(G)$ is connected.
    In this case, Cambie et al.~\cite{cambie2024optimally} investigated the diameter of $C_L(G)$. They conjectured that $diam(C_L(G)) \le n(G) + \mu(G)$ with $\mu(G)$ the size of a maximum matching of $G$ and proved several results towards this conjecture. 
    We answer to two of their open problems by proving the conjecture for two classes of graphs, namely subcubic graphs and complete multipartite graphs.
\end{abstract}

\section{Introduction}

Reconfiguration problems focus on finding transformations between two feasible solutions of a problem through a sequence of elementary steps, ensuring that each intermediate state is also a feasible solution. Recently, these problems have gained significant attention, with a wide range of problems being explored under the reconfiguration framework such as boolean satisfiability~\cite{gopalan2009connectivity}, independent sets~\cite{ bonamy2014reconfiguring, ito2020parameterized}, dominating sets~\cite{suzuki2016reconfiguration}, shortest paths~\cite{bonsma2012complexity, kaminski2011shortest} or graph coloring~\cite{bonsma2009finding, bonamy2014reconfiguration, bousquet2016fast}. 
The main questions in reconfiguration problems are (1) on the existence of such a transformation and (2) on the length of a shortest transformation, if one exists.

In this paper, we investigate the second question for the graph coloring reconfiguration problem (recoloring for short). Let $G$ be a graph and $k$ a non-negative integer. A \emph{$k$-coloring} of $G$ is a function that assigns a color from the set $\{1, \dots, k\}$ to each vertex of $G$. 
All colorings considered here are \emph{proper}, meaning no two adjacent vertices share the same color.
Recoloration questions can be rephrased as properties of the so-called \emph{$k$-recoloring graph} of $G$, denoted by $C_k(G)$, which is the graph whose vertices correspond to the $k$-colorings of $G$, and two colorings are adjacent if they differ on exactly one vertex. A key focus of this paper is to study the diameter of this recoloring graph. However, $C_k(G)$ may be disconnected or have infinite diameter unless $k$ is sufficiently large. A well-known sharp bound on $k$ such that $C_k(G)$ is connected is $\degen(G)+2$, with $\degen(G)$ the smallest $p$ such that $G$ is $p$-degenerate.

Cereceda~\cite{cereceda2007mixing} conjectured that if $k \ge \degen(G) + 2 $, the diameter of $C_k(G)$ is $O(n^2)$, where $n$ is the number of vertices in $G$. 
Recently, Bousquet and Heinrich~\cite{bousquet2022polynomial} provided the best known upper bound, showing that $\diam(C_k(G)) = O(n^{\degen(G) + 1})$. 
For larger $k$, specifically $k \ge \Delta(G) + 2$ where $\Delta(G)$ is the maximum degree of $G$, Cereceda~\cite{cereceda2007mixing} proved that the diameter is $O(n \cdot \Delta) = O(n^2)$. 
More recently, Cambie et al.~\cite{cambie2024optimally} improved this upper bound to $2n$ for $ k \ge \Delta(G) + 2$ and conjectured that it can be further reduced to $\lfloor 3n/2 \rfloor $.

In \cite{cambie2024optimally}, Cambie et al. also extended their investigation to more general contexts, namely list colorings and correspondence colorings. Given a \emph{list-assignment} $L$ for $G$, where each vertex $v$ is assigned a list of colors $L(v)$, a \emph{$L$-coloring} is a coloring where each vertex $v$ is colored from its respective list $L(v)$. The \emph{$L$-recoloring graph} of a graph $G$, denoted by $C_L(G)$ is constructed similarly as $C_k(G)$,  with vertices representing $L$-colorings of $G$ and edges connecting colorings that differ at exactly one vertex. 
In the context of list coloring, they investigated a precise bound for the diameter depending on the \emph{matching number} $\mu(G)$, which is the maximum size of a matching in $G$.
In particular, they conjecture the following:

\begin{conjecture}[Cambie et al.~\cite{cambie2024optimally}]\label{conj}
Let $G$ be a graph. If $L$ is a list-assignment such that $|L(v)| \ge d(v) + 2$ for every vertex $v$, then $\diam(C_L(G)) \le n(G) + \mu(G)$, with $\mu(G)$ the matching number of $G$.
\end{conjecture}

They provided significant evidence for their conjecture, proving it for cases where $|L(v)| \geq 2d(v)$ or where the diameter bound is relaxed to $n(G) + 2\mu(G)$. 
Additionally, they showed that the conjecture is tight for every graph and that it holds for trees, complete bipartite graphs, and cacti. 
In this paper, we answer the questions asked by Cambie et al. as to whether Conjecture~\ref{conj} is true for two additional classes of graphs: subcubic graphs and complete multipartite graphs.
For both classes, we prove the theorem by contradiction using a minimal counter-example.

\smallskip
\noindent
\textbf{Subcubic graphs.} \emph{Subcubic graphs} are graphs with maximum degree at most $3$. In \cite{cambie2024optimally}, Cambie et al. proved a version of Conjecture~\ref{conj} for subcubic graphs within the context of correspondence colorings. 
In this paper, we complement their result by proving Conjecture~\ref{conj} for subcubic graphs in the list-coloring setting:

\begin{restatable}{theorem}{subcubic}\label{thm:subcubic}
    Let $G$ be a subcubic graph. If a list-assignment $L$ satisfies $|L(v)| \ge d(v) + 2$ for every vertex $v$, then $\diam(C_L(G)) \le n(G) + \mu(G)$.
\end{restatable}

To achieve this result, we analyze the subgraphs induced by each color in both colorings. First, we show that one of these subgraphs contains a connected component $C$ of size $3$.
We then show that this component $C$ can be recolored in a few steps, ultimately leading to a contradiction.

\smallskip
\noindent
\textbf{Complete multipartite graphs.} Cambie et al. established that the conjecture holds for complete bipartite graphs. 
A natural generalization of these graphs is the \emph{complete multipartite} graph (also known as \emph{complete $r$-partite} graph), where the vertex set is partitioned into $r$ independent sets $V_1, \dots, V_r$, and each vertex in $V_i$ is adjacent to every vertex in $V_j$ for all $1 \leq i \neq j \leq r$. 
In this paper, we prove Conjecture~\ref{conj} for this broader class of graphs:

\begin{restatable}{theorem}{rpartite}\label{thm:rpartite}
    Let $G$ be a complete multipartite graph. If $L$ is a list-assignment such that $|L(v)| \ge d(v) + 2$ for every vertex $v$, then $\diam(C_L(G)) \le n(G) + \mu(G)$.
\end{restatable}

The proof is divided into two parts. In the first part, we assume a minimal counter-example contains a large independent set, then conclude following a method similar to the one used for complete bipartite graphs. In the second part, where the independent sets are more balanced, we derive a contradiction by counting the colors in the graph, particularly focusing on the number of colors appearing only once or twice in one of the coloring.

\smallskip

In the context of $k$-colorings, Cambie et al.\cite{cambie2024optimally} proved that for a subcubic graph $G$, $\diam(C_k) \ge n(G) + \mu(G)$ when $k \ge \Delta(G)+2$. Moreover, Zeven~\cite{zeven2023reconfiguration} proved the same lower bound on $\diam(C_k(G))$ for complete multipartite graphs when $k \ge \Delta(G)+2$.
Together with Theorem~\ref{thm:subcubic} and Theorem~\ref{thm:rpartite}, these results gives the exact value of the diameter of the $k$-reconfiguration graph of subcubic graphs and complete multipartite graphs.

\begin{corollary}
    Let $G$ be a subcubic graph. If $k \ge \Delta(G) + 2$, then $\diam(C_k(G)) = n(G) + \mu(G)$.
\end{corollary}

\begin{corollary}
    Let $G$ be a complete multipartite graph. If $k \ge \Delta(G) + 2$, then $\diam(C_k(G)) = n(G) + \mu(G)$.
\end{corollary}

We also remark that these results on subcubic and multipartite graphs cannot be obtained for a smaller number of colors (or list-assignment with smaller lists). Indeed, cycles whose length is divisible by $3$ and cliques are simple counter-examples: they are respectively subcubic and complete multipartite graphs, whose $(\Delta(G)+1)$-reconfiguration graph is not connected. More recently, Cambie et al.~\cite{cambie2025reconfiguration} studied in greater details the behavior of reconfiguration graphs when the number of available colors is further restricted.

\smallskip
\noindent
\textbf{Organization of the paper.} After providing some definitions in Section~\ref{sec:prem}, we prove Theorem~\ref{thm:subcubic} in Section~\ref{sec:sub} and then Theorem~\ref{thm:rpartite} in Section~\ref{sec:rpart}.

\section{Preliminaries}\label{sec:prem}

\noindent
\textit{Basic definitions and notations.} 
Let $G$ be a graph and $v$ be a vertex of $G$. The set of neighbors of $v$ in $G$ is denoted by $N(v)$ and its degree in $G$ is denoted by $d_G(v)$. 
We denote by $G[C]$ the subgraph of $G$ induced by $C$. 
A \emph{connected component $C$} of a graph is a maximal by inclusion subset of vertices such that $G[C]$ is connected. 

\medskip 
\noindent
\textit{Matchings.} A \emph{matching} $M$ of a graph $G$ is a set of pairwise disjoint edges. A matching $M$ \emph{covers} a vertex $v$ if $v$ is an endpoint of an edge of $M$. A matching is \emph{perfect} (resp. \emph{almost perfect}) if it covers every vertex (resp. every vertex but at most one).

\medskip

\noindent
\textit{Colorings.} 
In the figures, throughout the paper, we represent the initial coloring at the centers of the vertices, and the target coloring on the borders of the vertices, see Figure~\ref{fig:grapha}. 
Let $L$ be a list-assignment on $G$ and $\alpha, \beta$ be two $L$-colorings of $G$. 
The \emph{distance between $\alpha$ and $\beta$} in $C_L(G)$, denoted by $\dist_G(\alpha, \beta)$, is the size of a shortest path from $\alpha$ to $\beta$ in $C_L(G)$.
The \emph{color-shift digraph} $\overrightarrow{D_{\alpha, \beta}}$ of $G$ is the digraph whose vertices are exactly the vertices of $G$ and whose arcs go from a vertex $u$ to a vertex $v$ if and only if $\beta(u) = \alpha(v)$. We denote by $D_{\alpha, \beta}$ its underlying multigraph. 
See Figure~\ref{fig:digraph} and~\ref{fig:colorshift} for illustrations. 

We say a vertex $v$ \emph{has a conflict over a color} $c$ if there exists at least one vertex $u$ adjacent to $v$ such that $u$ is colored $c$ in one coloring and $v$ is colored $c$ in the other coloring. We then say that $u$ and $v$ are \emph{in conflict over color} $c$.
Hence, two vertices are in conflict if and only if they are adjacent in the color-shift digraph.

\begin{figure}[ht]
        \centering
        \tikzstyle{v1}=[circle, minimum size=3pt, scale=0.6, fill]
        \tikzstyle{v2}=[circle,draw, minimum size=12pt, line width=2]
        \tikzstyle{fleche}=[arrows = {-Latex[width=4pt, length=4pt]}]
        \tikzstyle{labell}=[text opacity=1, scale =1]
        \begin{subfigure}[b]{.3\textwidth}
        \centering
        \begin{tikzpicture}[scale=1]
        
        \nodc(a1)(1.2,0)[\colord, \colorg];
        \nodc(a2)(0,0)[\colorg, \colord];
        \nodc(a3)(-1,-0.7)[\colore, \colorf];
        \nodc(a4)(-1,0.7)[\colorf, \colorg];

        \draw(a1) to (a2);
        \draw(a3) to (a4);
        \draw(a4) to (a2);
        \draw(a3) to (a2);

        \end{tikzpicture}
        \caption{The graph $G$}\label{fig:grapha}
        \end{subfigure}
        \begin{subfigure}[b]{.3\textwidth}
        \centering
        \begin{tikzpicture}[scale=1]

        \nodc(a1)(1.2,0)[\colord, \colorg];
        \nodc(a2)(0,0)[\colorg, \colord];
        \nodc(a3)(-1,-0.7)[\colore, \colorf];
        \nodc(a4)(-1,0.7)[\colorf, \colorg];

        \draw[fleche](a1) to[bend left =20] (a2);
        \draw[fleche](a2) to[bend left =20] (a1);
        \draw[fleche](a3) to (a4);
        \draw[fleche](a4) to (a2);

        \end{tikzpicture}
        \caption{The digraph $\overrightarrow{D_{\alpha,\beta}}$}\label{fig:digraph}
        \end{subfigure}
        \begin{subfigure}[b]{.3\textwidth}
        \centering
        \begin{tikzpicture}[scale=1]

        \nodc(a1)(1.2,0)[\colord, \colorg];
        \nodc(a2)(0,0)[\colorg, \colord];
        \nodc(a3)(-1,-0.7)[\colore, \colorf];
        \nodc(a4)(-1,0.7)[\colorf, \colorg];

        \draw(a1) to[bend left= 20] (a2);
        \draw(a2) to[bend left =20] (a1);
        \draw(a3) to (a4);
        \draw(a4) to (a2);

        \end{tikzpicture}
        \caption{The multigraph ${D_{\alpha,\beta}}$}\label{fig:colorshift}
        \end{subfigure}
        
        \caption{Two colorings $\alpha$ (at the centers) and $\beta$ (on the borders) of a graph $G$, the color-shift digraph $\overrightarrow{D_{\alpha,\beta}}$ and multigraph $D_{\alpha, \beta}$ of $G$. } 
        \label{fig:exfig}
\end{figure}
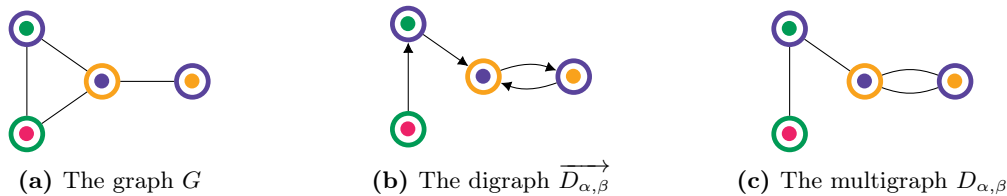 We say that a vertex $v$ has a \emph{free color $c$} in the colorings $\alpha$ and $\beta$ if $c \in L(v)$ and $c \notin \alpha(N(v)) \cup \beta(N(v))$.
Two colorings $\alpha$ and $\beta$ \emph{agree on} a set $S$ (or on a vertex when $S=\{v\}$) if $\alpha(v) = \beta(v)$ for every vertex $v$ in $S$.
Hence, if a vertex $v$ has a free color $c$, we can recolor $v$ with $c$ in both colorings to get colorings that agree on $v$.
Here, instead of constructing a sequence from a current coloring $\alpha$ to a target coloring $\beta$, we will construct two distinct sequences from $\alpha$ and $\beta$ to a common coloring~$\gamma$.

For simplicity, for any graph $G$, we define $\alpha(G) := \alpha(V(G))$.
We say that a color $c$ \emph{appears $k$ times in a set $U$ for a coloring $\alpha$} if $k$ vertices of $U$ are colored $c$ in the coloring $\alpha$. 
When the subset of vertices is ommited, we consider all vertices of the graph.

\medskip

\noindent
\textit{Minimal counter-example.} We prove Theorems~\ref{thm:subcubic} and~\ref{thm:rpartite} by contradiction. A \emph{minimal counter-example} $G$ is a graph that does not satisfy the theorem and such that every graph with fewer vertices satisfies the theorem. 
Hence, a minimal counter-example $G$ admits a list-assignment $L$ and two $L$-colorings $\alpha$ and $\beta$ whose distance is more than $n(G) + \mu(G)$.
To reach a contradiction, we will always use the same method: we construct new colorings $\gamma$ and $\zeta$ from $\alpha$ and $\beta$ such that $\gamma$ and $\zeta$ agree on a set of vertices $S$ of $G$. 
Then, we obtain a new graph $G'$ by removing $S$ from $G$ and a new list-assignment $L'$ from $L$ by removing the colors of the vertices of $S$ from the lists of their respective neighborhood. 
In this case, we say $L'$ is obtained by \emph{restricting} $L$ to $G'$.
Finally, using that we can recolor $G'$ from $\gamma$ to $\zeta$ using $n(G') + \mu(G')$ steps (by minimality), we prove that we can recolor $G$ from $\alpha$ to $\beta$ with $n(G) + \mu(G)$ steps.

\section{Subcubic graphs}\label{sec:sub}

\subcubic*

The rest of this section is devoted to proving Theorem~\ref{thm:subcubic}.
Let $G$ be a minimal counter-example.
Consider a list-assignment $L$ of $G$ such that $|L(v)| \ge d(v) +2$ for every vertex $v$ and two $L$-colorings $\alpha$ and $\beta$ such that $\dist_G(\alpha, \beta) > n(G) + \mu(G)$.
For each color $c$, we define the \emph{color graph} $G_c$ as the graph whose vertices are those in $V(G)$ colored $c$ in either $\alpha$ or $\beta$ and where two vertices are linked by an edge if they are in conflict over $c$ in $\alpha$ and $\beta$. 
Note that the union of the graphs $G_c$ for every $c$ forms the multigraph $D_{\alpha,\beta}$.
We conclude with several lemmas depending on the size of connected components of the color graphs.

\begin{lemma}\label{lem:cc1}
    No color graph contains an isolated vertex $v$.
\end{lemma}

\begin{proof}
Assume by contradiction that some color graph $G_c$ contains an isolated vertex $v$.
Then, $v$ is colored $c$ in $\alpha$ or $\beta$, say the latter by symmetry.
And since it has no neighbors in $G_c$, $v$ has no conflict with its neighbors over the color $c$, i.e. no neighbor of $v$ is colored $c$ in $\alpha$. 
Hence, we define $\gamma$ as the coloring obtained from $\alpha$ by recoloring $v$ with $c$ in a single step. 
Note that $\gamma$ and $\beta$ agree on $v$. 
Let $G' = G[V \setminus \{ v \}]$ be the graph obtained from $G$ by removing $v$ and $L'$ be the restriction of $L$ to $G'$.
By minimality of $G$, we can recolor $G'$ from $\gamma$ to $\beta$ using at most $n(G') + \mu(G')$ steps.
Hence, we can recolor $G$ from $\alpha$ to $\beta$ using at most $1 + n(G') + \mu(G') \le n(G) + \mu(G)$ steps, a contradiction.
\end{proof}

Hence, from now on, we will assume that the color graphs do not contain isolated vertices.
The rest of the proof consists in first finding a connected component of size at least $3$ in a color graph of $G$, and then proving that we can recolor this connected component.

\begin{lemma}~\label{lem:cc2}
There exists a color $c$ such that $G_c$ contains a connected component of size at least $3$.
\end{lemma}

\begin{proof}
Assume by contradiction that, for every color $c$, $G_c$ only contains connected components of size at most $2$.
By Lemma~\ref{lem:cc1}, every connected component of the color graphs contains exactly two vertices. 
Thus, the color-shift digraph $\overrightarrow{D_{\alpha, \beta}}$ only contains vertices of indegree $1$ and outdegree $1$, that is the union of vertex-disjoint directed cycles (possibly being bidirected arcs).

Let $C$ be a cycle of $\overrightarrow{D_{\alpha, \beta}}$ and $u,v,w$ be three consecutive vertices of $C$, with possibly $u = w$ if $C$ is a bidirected arc.
Let $\gamma$ be the coloring obtained from $\alpha$ as follows: 
first recolor $w$ with a color of $L(w) \setminus (\alpha(N(w)) \cup \{ \beta(v) \})$ (which is possible since $|L(w)| \ge d(w) + 2$).
Then we can recolor $v$ with $\beta(v) = \alpha(w)$ and finally recolor $u$ with $\beta(u) = \alpha(v)$. See Figure~\ref{fig:cc2} for an illustration.

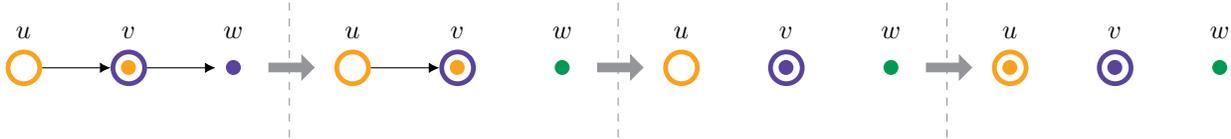
\begin{figure}[ht]
    \centering
        \tikzstyle{v1}=[circle, minimum size=3pt, scale=0.6, fill]
        \tikzstyle{v2}=[circle,draw, minimum size=12pt, line width=2]
        \tikzstyle{fleche}=[arrows = {-Latex[width=4pt, length=4pt]}]
        \tikzstyle{labell}=[text opacity=1, scale =1]
        \tikzstyle{garrow}=[single arrow, draw, 
      minimum width =1mm, single arrow head extend=3pt, inner sep=0.5mm,
      minimum height=6mm, fill=\colora, \colora]
        \begin{tikzpicture}[scale=0.93]

        \nodc(a1)(1.5,0)[\colorg, none];
        \nodc(a2)(0,0)[\colord, \colorg];
        \nodc(a3)(-1.5,0)[none, \colord];

        \draw[fleche] (a3) to (a2);
        \draw[fleche] (a2) to (a1);

        \node (h0) at (1.5, 0.5) [labell]{$w$};
        \node (h0) at (0, 0.5) [labell]{$v$};
        \node (h0) at (-1.5, 0.5) [labell]{$u$};

        \draw[dashed, \colora] (2.3, -1) to  (2.3, 1);
        \node (f1) at (2.3,-0) [garrow, fill=\colora, \colora]{};
        
        \tikzset{xshift=4.7cm}

        \nodc(a1)(1.5,0)[\colorf, none];
        \nodc(a2)(0,0)[\colord, \colorg];
        \nodc(a3)(-1.5,0)[none, \colord];

        \draw[fleche] (a3) to (a2);

        \node (h0) at (1.5, 0.5) [labell]{$w$};
        \node (h0) at (0, 0.5) [labell]{$v$};
        \node (h0) at (-1.5, 0.5) [labell]{$u$};

        \draw[dashed, \colora] (2.3, -1) to  (2.3, 1);
        \node (f1) at (2.3,-0) [garrow, fill=\colora, \colora]{};
        
        \tikzset{xshift=4.7cm}

        \nodc(a1)(1.5,0)[\colorf, none];
        \nodc(a2)(0,0)[\colorg, \colorg];
        \nodc(a3)(-1.5,0)[none, \colord];

        \node (h0) at (1.5, 0.5) [labell]{$w$};
        \node (h0) at (0, 0.5) [labell]{$v$};
        \node (h0) at (-1.5, 0.5) [labell]{$u$};

        \draw[dashed, \colora] (2.3, -1) to  (2.3, 1);
        \node (f1) at (2.3,-0) [garrow, fill=\colora, \colora]{};

        \tikzset{xshift=4.7cm}

        \node (h0) at (1.5, 0.5) [labell]{$w$};
        \node (h0) at (0, 0.5) [labell]{$v$};
        \node (h0) at (-1.5, 0.5) [labell]{$u$};

        \nodc(a1)(1.5,0)[\colorf, none];
        \nodc(a2)(0,0)[\colorg, \colorg];
        \nodc(a3)(-1.5,0)[\colord, \colord];

        \end{tikzpicture}
        \caption{Recoloring sequence of Lemma~\ref{lem:cc2} for three consecutive vertices $u,v$ and $w$ in $\overrightarrow{D_{\alpha, \beta}}$. We only represent some arcs of the color-shift digraph. } 
        \label{fig:cc2}
\end{figure} 
We get $\gamma$ from $\alpha$ in three steps and $\gamma,\beta$ agree on two adjacent vertices $u$ and $v$. 
Let $G' := G[V\setminus\{u,v\}]$ and let $L'$ be the restriction of $L$ to $G'$.
By minimality of $G$, we can recolor $G'$ from $\gamma$ to $\beta$ in at most $n(G') + \mu(G')$ steps.
Since $u$ and $v$ are adjacent in $G$, we can add $uv$ to any matching of $G'$ to obtain a matching of $G$, so $\mu(G') \le \mu(G) - 1$ . 
Hence, we can recolor $G$ from $\alpha$ to $\beta$ in at most $n(G') + \mu(G') + 3 \le n(G) - 2 + \mu(G) - 1 + 3 \le n(G) + \mu(G)$ steps, a contradiction. 
\end{proof}

Let $c$ be a color such that $G_c$ admits a connected component $V_0$ of maximum size and let $G_0 := G_c[V_0]$. 
By Lemma~\ref{lem:cc2}, $G_0$ contains at least $3$ vertices. 
Since two adjacent vertices cannot be colored the same, $G_0$ is bipartite with a bipartition $(A,B)$ where $A$ is the set of vertices in $V_0$ colored $c$ in $\alpha$ and $B$ the set of vertices in $V_0$ colored $c$ in $\beta$.

We will conclude the proof using a \emph{vertex cover} of $G_0$, that is a set $W$ of vertices in $V_0$ such that every edge of $G_0$ has at least one endpoint in $W$.
We first prove the existence of a minimum vertex cover $W$ of $G_0$ that does not contain vertices of degree at most $1$ (Lemma~\ref{cl:vcex}). 
Then, we will use this vertex cover to recolor the vertices in $V_0$ and conclude by minimality of $G$ (Lemma~\ref{cl:agreeV0}).

\begin{lemma}\label{cl:vcex}
    Let $G$ be a connected graph on at least $3$ vertices. Then $G$ has a vertex cover $W$ with no vertex of degree $1$.
\end{lemma}
\begin{proof}
Let $W'$ be a minimum vertex cover of $G$.
By minimality of $W'$, for every vertex $v$ of $W'$, there exists an edge $e_v$ whose only endpoint in $W'$ is $v$. 
Then, construct $W$ from $W'$ by exchanging every vertex of $W'$ of degree $1$ in $G$ with its neighbor. 
Since $G$ is connected and has at least $3$ vertices, the set $W$ does not contain a vertex of degree $1$ and is still a vertex cover. 
The conclusion follows.
\end{proof}

\begin{lemma}\label{cl:agreeV0}
    Let $W$ be a vertex cover of $G_0$ containing no vertices of degree $1$. Then, we can recolor $\alpha$ and $\beta$ into two colorings $\alpha^*$ and $\beta^*$ that agree on $V_0$ in at most $|W| + |V_0|$ steps. 
\end{lemma}

\begin{proof}
The proof will be in two steps.
First, we recolor the vertices of $W$ to make the colorings agree on $W$ using colors distinct from $c$. 
Then, we recolor each vertex in $V_0 \setminus W$ with $c$ in $\alpha$ and $\beta$ in a single step since $W$ covers $G_0$. 

Denote by $\alpha', \beta'$ the current colorings and say we recolor a vertex $v$ of $W$.
Recall that $c$ is the color shared by all vertices of $V_0$.
We transform $\alpha'$ and $\beta'$ such that $c$ is not a color of $v$ for both resulting colorings and the resulting colorings agree on $v$.

If $v$ has a free color $c^*$ in the colorings $\alpha'$ and $\beta'$, then we recolor $v$ in both $\alpha'$ and $\beta'$ with $c^*$ in at most two steps.
Otherwise, we can assume that $v$ has no free color in $\alpha', \beta'$.
Hence, $\alpha'(N(v)) \cup \beta'(N(v))$ contains at least $d_G(v) + 2$ distinct colors, including $\alpha'(v)$ and $\beta'(v)$.
Since $2d_G(v)$ colors appear in the neighborhood of $v$ for $\alpha'$ and $\beta'$, the colors $\alpha'(v)$ and $\beta'(v)$ appear at most $d_G(v) \le 3$ times in $N(v)$ for $\beta'$ and $\alpha'$ in total.
Since $v \in G_c$, we can assume by symmetry that $\alpha'(v) = c$.
By construction of $W$, $v$ has degree at least $2$ in $G_0 = G_c[V_0]$, so $\alpha'(v)$ appears at least twice in $N(v)$ for $\beta'$.
And since $\beta(v)'$ is not free for $v$ in $\alpha', \beta'$, the color $\beta'(v)$ appears at least once in $N(v)$ for $\alpha'$.
Hence, the degree of $v$ in $\overrightarrow{D_{\alpha', \beta'}}$ is exactly $3$: it has indegree $2$ and outdegree $1$.
Let $vu$ be the outgoing arc of $v$ in $\overrightarrow{D_{\alpha', \beta'}}$.
We transform $\alpha'$ in at most two steps: first recolor $u$ from $\beta'(v)$ with any color in $L(u) \setminus \alpha'(N[u])$ then recolor $v$ in $\alpha'$ from $c$ to $\beta'(v)$. 

In both cases, we recolor at most twice the vertices of $W$. 
The resulting colorings agree on $W$ and no vertex of $W$ is colored $c$.
Moreover, the previous color and new color of $u$ are not $c$, so the vertices of $V_0 \setminus W$ are still colored $c$ in one of the colorings and there is no new vertex colored $c$.
Hence, since $W$ covers $G_0$, the vertices of $V_0 \setminus W$ have no conflict over $c$.
We can recolor every vertex of $V_0 \setminus W$ in a single step such that its color is $c$ in both resulting colorings $\alpha^*$ and $\beta^*$. 
See Figure~\ref{fig:cc3} for an illustration of the recoloring sequence.
In total, we use at most $2|W| + |V_0 \setminus W| = |W| + |V_0|$ steps to obtain $\alpha^*, \beta^*$ from $\alpha, \beta$.
\end{proof}

\begin{figure}[ht]
        \centering
        \tikzstyle{v1}=[circle, minimum size=3pt, scale=0.6, fill]
        \tikzstyle{v2}=[circle,draw, minimum size=12pt, line width=2]
        \tikzstyle{fleche}=[arrows = {-Latex[width=4pt, length=4pt]}]
        \tikzstyle{labell}=[text opacity=1, scale =1]
        \tikzstyle{bigarrow}=[>= {Triangle[scale=0.5]}]
        \tikzstyle{garrow}=[single arrow, draw, 
      minimum width =2mm, single arrow head extend=3pt, inner sep=1mm,
      minimum height=8mm, fill=\colora, \colora]
        \begin{tikzpicture}[use Hobby shortcut,scale=1,transform shape]

        \nodc(a1)(1.5,0)[\colorg, \colord];
        \nodc(a2)(1.5,0.7)[\colorg, \colorb];
        \nodc(a3)(1.5,-0.7)[\colorg, \colorc];
        \nodc(b1)(0,0.35)[\colord, \colorg];
        \nodc(b2)(0,-0.35)[\colorf, \colorg];
        \nodc(b3)(0,1.05)[\colorf, \colorg];
        \nodc(b4)(0,-1.05)[\colorb, \colorg];
        
        \draw[fleche] (b1) to (a1);
        \draw[fleche] (b2) to (a1);
        \draw[fleche] (b3) to (a1);
        \draw[fleche] (b1) to (a2);
        \draw[fleche] (b2) to (a3);
        \draw[fleche] (b3) to (a3);
        \draw[fleche] (b4) to (a3);
        \draw[fleche] (a1) to (b1);

        \node (h0) at (0.75, 1.55) [labell]{$V_0$};

        \node (h1) at (0.8,0) [labell]{};
        \node (h2) at (2.15,-0.3) [labell]{$W$};

        \begin{pgfonlayer}{background}
            \fill[\colora, opacity=0.4] \convexpath{a1,a3,h1,b1}{10pt};
        \end{pgfonlayer}

        \draw[dashed, \colora] (3, -1) to  (3, 1);
        \node (f1) at (3,0) [garrow]{};

        \tikzset{xshift=4.5cm}     

        \nodc(a1)(1.5,0)[\colore, \colore];
        \nodc(a2)(1.5,0.7)[\colorg, \colorb];
        \nodc(a3)(1.5,-0.7)[\colorc, \colorc];
        \nodc(b1)(0,0.35)[\colord, \colord];
        \nodc(b2)(0,-0.35)[\colorf, \colorg];
        \nodc(b3)(0,1.05)[\colorf, \colorg];
        \nodc(b4)(0,-1.05)[\colorb, \colorg];

        \node (h0) at (0.75, 1.55) [labell]{$V_0$};

        \node (h1) at (0.8,0) [labell]{};
        
        \node (h2) at (2.15,-0.3) [labell]{$W$};
        
        \begin{pgfonlayer}{background}
            \fill[\colora, opacity=0.4] \convexpath{a1,a3,h1,b1}{10pt};
        \end{pgfonlayer}

        \draw[dashed, \colora] (3, -1) to  (3, 1);
        \node (f1) at (3,0) [garrow]{};

        \tikzset{xshift=4.5cm}

        \nodc(a1)(1.5,0)[\colore, \colore];
        \nodc(a2)(1.5,0.7)[\colorg, \colorg];
        \nodc(a3)(1.5,-0.7)[\colorc, \colorc];
        \nodc(b1)(0,0.35)[\colord, \colord];
        \nodc(b2)(0,-0.35)[\colorg, \colorg];
        \nodc(b3)(0,1.05)[\colorg, \colorg];
        \nodc(b4)(0,-1.05)[\colorg, \colorg];

        \node (h0) at (0.75, 1.55) [labell]{$V_0$};
        
        \node (h1) at (0.8,0) [labell]{};
        
        \node (h2) at (2.15,-0.3) [labell]{$W$};

        \begin{pgfonlayer}{background}
            \fill[\colora, opacity=0.4] \convexpath{a1,a3,h1,b1}{10pt};
        \end{pgfonlayer}

        \end{tikzpicture}
        \caption{Reconfiguration sequence of Lemma~\ref{cl:agreeV0}. First, we make the colorings agree on $W$ (in gray here) on colors distinct from $c$ (violet here), then we recolor the vertices of $V_0 \setminus W$ with $c$. We only represent arcs of the color-shift digraph that are between vertices of $V_0$. } 
        \label{fig:cc3}
\end{figure}
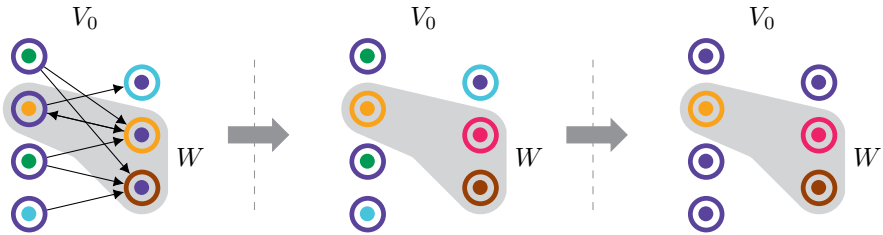 

Now, we conclude the proof of Theorem~\ref{thm:subcubic}.
Let $W$ be a vertex cover of $G_0$ obtained from Lemma~\ref{cl:vcex}
Since $W$ is a minimum vertex cover of the bipartite graph $G_0$, K\H{o}nig's theorem ensures that $G_0$ has a matching of size $|W|$. So $|W| +  \mu(G - V_0) \le \mu(G)$.
By Lemma~\ref{cl:agreeV0}, we construct two colorings $\alpha^*$ and $\beta^*$ from $\alpha$ and $\beta$ that agree on $V_0$ in at most $|W|+|V_0|$ steps. 
By minimality of $G$ applied to $G-V_0$ with the list-assignment $L'$ obtained by restricting $L$ to $G-V_0$, we obtain the following contradiction:
\[
\begin{aligned}
    \dist_G(\alpha, \beta) &\le |V_0| + |W| + d_{G - V_0}(\alpha^*, \beta^*) \\
                    &\le |V_0| + |W| + n(G - V_0) + \mu(G - V_0) \\
                    &= n(G) + |W| +  \mu(G - V_0) \\
                    &\le n(G) + \mu(G).
\end{aligned}
\]

\section{Complete multipartite graphs}\label{sec:rpart}

\rpartite*

We prove Theorem~\ref{thm:rpartite} by contradiction.
Let $G$ be a minimal counter-example and let $I_1, \ldots, I_m$ be the maximal independent sets of $G$ ordered in decreasing size.
Consider a list-assignment $L$ of $G$ and two $L$-colorings $\alpha$, $\beta$ such that $ \dist_G(\alpha, \beta) > n(G) + \mu(G)$.

\smallskip

We distinguish two cases.
In the first case, we assume that $|I_1| \ge \sum_{j>1}|I_j|$.
Then we adapt the proof technique of Theorem~20 of~\cite{cambie2024optimally} which proves Conjecture~\ref{conj} for complete bipartite graphs. 
In the second case, since no part is too large, $G$ has an almost perfect matching, which allows us to easily obtain a graph with a smaller maximum matching when removing vertices. By upper bounding the number of colors, we reach a contradiction on the colors of $I_3$.
Before giving the proof, let us start with simple useful observations:

\begin{lemma}\label{lem:aeqb}
    Every vertex has conflicts in both colorings $\alpha$ and $\beta$. 
\end{lemma}
\begin{proof}
By symmetry, assume a vertex $v$ has no conflict over $\beta(v)$.
Then, in $\alpha$, we can recolor $v$ with $\beta(v)$. 
By minimality of $G$, we have the following contradiction :
\[\dist_G(\alpha, \beta) \le n(G - v) + \mu(G - v) + 1 \le n(G) + \mu(G).\qedhere\]
\end{proof}

Thus, for every set $I_i$ of $G$, we deduce that $\beta(I_i) \subseteq \alpha(G - I_i)$ and $\alpha(I_i) \subseteq \beta(G - I_i)$.
In particular, $\alpha(G) = \beta(G)$.
Also, since the independent sets are complete to each other, matchings in $G$ will have the following property:

\begin{lemma}\label{cl:mudecrease}
    Let $V_i \subseteq I_i$ and $V_j \subseteq I_j$ with $i \neq j$.
    \[ \mu(G - V_i - V_j) \leq \mu(G) - \min(|V_i|, |V_j|) \]
\end{lemma}
\begin{proof}
    Consider a maximum matching $M$ of $G - V_i - V_j$. 
    Adding to $M$ a matching of size $\min(|V_i|, |V_j|)$ between $V_i$ and $V_j$ to $M$ gives a matching of $G$. 
    Thus, $\mu(G - V_i - V_j) + \min(|V_i|, |V_j|) \leq \mu(G)$, which gives the result. 
\end{proof}

\subsection{Large independent set}

Assume that $|I_1| \ge \sum_{j>1}|I_j|$.
Thus, every maximum matching of $G$ covers all the vertices of $V(G) - I_1$ and has size $|V(G) - I_1|$.
Hence, for every vertex set $U' \subseteq V(G) - I_1$, $\mu(G - U') = \mu(G) - |U'|$ since the graph is a complete multipartite graph.

We first prove that there is no conflict in $G$ between two small independent sets, and thus every conflict of $G$ involves a vertex of $I_1$. Then, we will follow the proof technique of Conjecture~\ref{conj} for complete bipartite graph given in~\cite{cambie2024optimally}.

\begin{lemma}\label{lem:conflict}
    Let $v$ be a vertex of $I_j$ with $j \ge 2$. 
    Then, $v$ is only in conflict with vertices of $I_1$.
\end{lemma}
\begin{proof}
    Assume by contradiction that $v$ has a conflict with a vertex that is not in $I_1$ over the color $c$. 
    Let $A$ be the set of vertices colored $c$ in $\alpha$ and $B$ the set of vertices colored $c$ in $\beta$. 
    Since $G$ is a complete multipartite graph, $A$ (resp. $B$) is included in an $I_a$ (resp. $I_b$) for some $a,b>1$. 
    By symmetry, we can assume that $|A| \le |B|$.
    Let $\gamma$ be the coloring obtained from $\alpha$ as follows.
    Recolor the vertices $u \in A$ with another color, which is possible since $|L(u)| \ge d(u) + 2 > |\alpha(N(u)) \cup \{c\}|$.
    Then, recolor the vertices in $B$ with $c$ so that $\gamma$ and $\beta$ agree on $B$.
    Let  $B$ be the list-assignment obtained from restricting $L$ to $G-B$ (i.e. delete $c$ from the lists of the neighbors of $B$). 
    By minimality of $G$, the distance from $\gamma$ to $\beta$ in $G - B$ is at most $n(G - B) + \mu(G - B) = n(G) + \mu(G) - 2|B|$. 
    Thus, we have the following contradiction: 
    \[ \dist_G(\alpha, \beta) \le n(G) + \mu(G) - 2|B| + |A|+|B| \leq  n(G) + \mu(G).  \qedhere\]
\end{proof}

By Lemma~\ref{lem:conflict}, we have that $\beta(G - I_1) \subseteq \alpha(I_1)$ and $\alpha(G - I_1) \subseteq \beta(I_1)$.
Since we also have that $\alpha(I_1) \subseteq \beta(G - I_1)$ and $\beta(I_1) \subseteq \alpha(G - I_1)$ by Lemma~\ref{lem:aeqb}, we deduce that $\beta(G - I_1) =\alpha(I_1)$ and $\alpha(G - I_1) = \beta(I_1)$.
We distinguish two cases.

Assume first that a color $c$ appears only once in $I_1$ for $\alpha$, say on a vertex $w$.
Let $\gamma$ be the coloring obtained from $\alpha$ as follows: recolor $w$ with a color distinct from $\alpha(G - I_1) \cup \{c\}$ and recolor every $u \in \beta^{-1}(c)$ with $c$.
Let $G'$ be the graph obtained from $G$ by removing $\beta^{-1}(c) \subseteq G - I_1$ and $L'$ the list-assignment obtained by restricting $L$ to $G'$.
By minimality, we can recolor $G'$ from $\gamma$ to $\beta$ in at most $n(G')+\mu(G')$ steps.
Since $|\beta^{-1}(c)|\ge 1$ by Lemma~\ref{lem:aeqb}, we have a transformation from $\alpha$ to $\beta$ of size at most: \[ 1 + |\beta^{-1}(c)| + n(G') + \mu(G') = 1 + |\beta^{-1}(c)| + (n(G) - |\beta^{-1}(c)|) + (\mu(G) - |\beta^{-1}(c)|) \le n(G) + \mu(G) \]
which gives a contradiction. 

Thus, we can assume that every color of $\alpha(I_1)$ appears at least twice in $I_1$. Thus, $|I_1| \ge 2|\alpha(I_1)|$.
By symmetry, we also have that $|I_1| \ge 2|\beta(I_1)|$.
Hence, $|I_1| \ge |\alpha(I_1)| + |\beta(I_1)|$.
For each $u \in I_i$ (with $2 \le i \le r$), we have that: \[ |L(u)| \ge \sum_{j \neq i}|I_j| + 2 \ge  |\alpha(I_1)| + |\beta(I_1)|  + \sum_{j \notin\{1,i\}}|I_j| + 2 \]
Hence we can obtain $\beta$ from $\alpha$ as follows. 
Starting from $\gamma=\alpha$, we successively recolor each $u \in I_i \subseteq G - I_1$ with a color in $L(u) \setminus (\gamma(G-I_i) \cup \beta(I_1))$.
Such a recoloring is possible since $\gamma(I_1)=\alpha(I_1)=\beta(G - I_1)$ at each step. 
Moreover this recoloring creates no conflict on $G-I_1$.
Then, recolor each $w \in I_1$ with $\beta(w)$. Finally, recolor each $u$ in $G - I_1$ with $\beta(u)$. 
The transformation needs at most $2|G - I_1| + |I_1| = n(G) + \mu(G)$ steps, a contradiction.

\subsection{Almost perfect matching}

From the previous section, we can assume $|I_i| < |V(G) \setminus I_i|$ for every $1 \le i \le r$.
Hence, $G$ has at least three independent sets, that is $m\ge 3$. 
The graph $G$ also satisfies the following property.

\begin{lemma}~\label{lem:apm}
    The graph $G$ has an almost-perfect matching.
\end{lemma}
\begin{proof}
Let $M$ be a maximum matching of $G$. Observe that for every $i$, $M$ contains an edge $u_iv_i$ without endpoints in $I_i$ (otherwise, $|I_i| \ge |V(G) \setminus I_i|$).
Suppose that there are two vertices $u,v$ that are not incident to edges of $M$.
By maximality of $M$, $u,v$ cannot be adjacent, so they lie in some $I_i$. 
But then, we can replace $u_iv_i$ by $uu_i$ and $vv_i$ in $M$.
Hence, for every pair of vertices, at least one of them is the endpoint of an edge of $M$.
So $M$ is an almost perfect matching.    
\end{proof}

To prove Theorem~\ref{thm:rpartite}, we will first bound by $\frac 23 n(G)$ the total number of distinct colors in $\alpha(G)$ (and thus in $\beta(G)$ by Lemma~\ref{lem:aeqb}). 
Then, we use this bound to show that every vertex of $I_3$ has a free color. 
These free colors allow us to obtain colorings that agree on $I_3$ and its neighborhood in the color-shift digraph in a few steps, which brings a contradiction. 

If every color appears at least twice on $G$ in $\alpha$, then we would have our desired bound on the number of colors in $\alpha(G)$ (and even better). 
This actually may not be true, but we will control the number of colors appearing only once in a coloring. 
Let $A := \{ w \in V(G), \ |\alpha^{-1}(\alpha(w))| = 1 \}$ and $B := \{ w \in V(G), \ |\beta^{-1}(\beta(w))| = 1 \}$ be the sets of vertices whose color appears only once in $\alpha$ and $\beta$ respectively. 
Our first objective is to prove that every vertex in $A\cup B$ has a unique \emph{cousin}, that is another vertex in $A \cup B$ having its unique color in the other coloring. We start by proving that a color appears only once for $\alpha$ if and only if it also appears only once for $\beta$:

\begin{lemma}\label{prop:matchingab}
    If $w \in A$, then $\alpha(w)$ appears only once in $\beta$. Symmetrically, if $u \in B$, then $\beta(u)$ appears only once in $\alpha$.
\end{lemma}
\begin{proof}
    Since $\alpha(G) = \beta(G)$, there is at least one vertex colored with $\alpha(w)$ in $\beta$.
    Assume by contradiction that $B_w:=\beta^{-1}(\alpha(w))$ contains at least two vertices.
    Let $\gamma$ be the coloring obtained from $\alpha$ by first recoloring $w$ with a color in $L(w) \setminus \alpha(N[w])$, which is possible since $|L(w)| \geq d(w) + 2$. 
    And then recoloring every vertex in $B_w$ into $\alpha(w)$ (see Figure~\ref{fig:matchingab}).
    Let $G'$ be the graph obtained from $G$ by removing $B_w$ and $L'$ the list-assignment obtained by restricting $L$ to $G'$.
    By minimality of $G$, we can recolor $G'$ from $\gamma$ to $\beta$ in at most $n(G') + \mu(G')$ steps. 
    Since $B_w$ contains at least two vertices and $G$ has an almost-perfect matching, $\mu(G') \le \mu(G) - 1$.
    Thus, we can recolor $G$ from $\alpha$ to $\beta$ in at most $1 + |B_w| + n(G') + \mu(G') \leq 1 + |B_w| + (n(G) - |B_w|) + (\mu(G) - 1) = n(G) + \mu(G)$ steps, a contradiction. By symmetry, we also have that $|\alpha^{-1}(\beta(u))| = 1$ for every $u \in B$.
\end{proof}

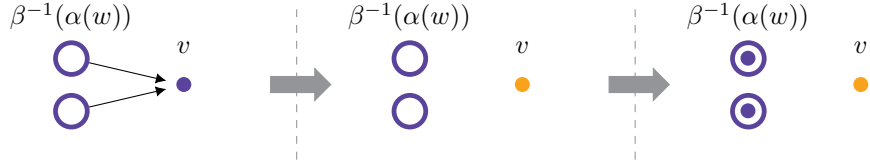
\begin{figure}[ht]
        \centering
        \tikzstyle{v1}=[circle, minimum size=3pt, scale=0.6, fill]
        \tikzstyle{v2}=[circle,draw, minimum size=12pt, line width=2]
        \tikzstyle{fleche}=[arrows = {-Latex[width=4pt, length=4pt]}]
        \tikzstyle{labell}=[text opacity=1, scale =1]
        \tikzstyle{bigarrow}=[>= {Triangle[scale=0.5]}]
        \tikzstyle{garrow}=[single arrow, draw, 
      minimum width =2mm, single arrow head extend=3pt, inner sep=1mm,
      minimum height=8mm, fill=\colora, \colora]
        \begin{tikzpicture}[scale=1]

        \nodc(a1)(1.5,0)[\colorg, none];
        \nodc(a2)(0,0.35)[none, \colorg];
        \nodc(a3)(0,-0.35)[none, \colorg];

        \draw[fleche] (a3) to (a1);
        \draw[fleche] (a2) to (a1);

        \node (h0) at (1.5, 0.5) [labell]{$w$};
        \node (h0) at (0, 0.9) [labell]{$\beta^{-1}(\alpha(w))$};

        \draw[dashed, \colora] (3, -1) to  (3, 1);
        \node (f1) at (3,0) [garrow]{};

        \tikzset{xshift=4.5cm}

        \nodc(a1)(1.5,0)[\colord, none];
        \nodc(a2)(0,0.35)[none, \colorg];
        \nodc(a3)(0,-0.35)[none, \colorg];

        \node (h0) at (1.5, 0.5) [labell]{$w$};
        \node (h0) at (0, 0.9) [labell]{$\beta^{-1}(\alpha(w))$};

        \draw[dashed, \colora] (3, -1) to  (3, 1);
        \node (f1) at (3,0) [garrow]{};

        \tikzset{xshift=4.5cm}

        \nodc(a1)(1.5,0)[\colord, none];
        \nodc(a2)(0,0.35)[\colorg, \colorg];
        \nodc(a3)(0,-0.35)[\colorg, \colorg];

        \node (h0) at (1.5, 0.5) [labell]{$w$};
        \node (h0) at (0, 0.9) [labell]{$\beta^{-1}(\alpha(w))$};

        \end{tikzpicture}
        \caption{Reconfiguration sequence of Lemma~\ref{prop:matchingab} to obtain $\gamma$. The color of $w$ in $\alpha$ appears only once in $\alpha$ but at least twice in $\beta$. We only represent some arcs of the color-shift digraph. } 
        \label{fig:matchingab}
\end{figure}

Hence, for every $w$ in $A$, there exists a unique vertex $u$ in $B$ such that $\alpha(w) = \beta(u)$. 
We say that $w$ and $u$ are \emph{cousins}.
Now, we show that every vertex of $A \cup B$ has a unique cousin by proving that $A$ and $B$ are disjoint (since only a vertex in $A \cap B$ can have two cousins).

\begin{lemma}\label{prop:emptyab}
    $A \cap B = \emptyset$
\end{lemma}
\begin{proof}
    Assume by contradiction that there exists $v \in A \cap B$. 
    By Lemma~\ref{prop:matchingab}, there exist $u$ in $B$ and $w$ in $A$ (with possibly $u=w$) such that $\alpha(v) = \beta(u)$ and $\beta(v) = \alpha(w)$. 
    Construct a new coloring from $\alpha$ as follows.
    Recolor $w$ with a color in $L(w) \setminus \alpha(N[w])$.
    Then, recolor $v$ with $\beta(v) = \alpha(w)$, and finally, recolor $u$ with $\beta(u) = \alpha(v)$.
    Restrict $L$ to $G - u - v$. 
    By minimality of $G$, we can recolor $G$ from $\alpha$ to $\beta$ in at most $3 + n(G - u - v) + \mu(G - u - v) \le 3 + (n(G) - 2) + (\mu(G) - 1) \le n(G) + \mu(G)$, a contradiction.
\end{proof}

By Lemmas~\ref{prop:matchingab} and~\ref{prop:emptyab}, each vertex $w$ in $A\cup B$ has a unique cousin, denoted by $\widetilde{w}$.
Note that this gives a first bound on the number of colors in $\alpha(G)$.
Indeed, every vertex in $A$ has a cousin $\widetilde{a}$ which is not in $A$, so $|A| \le \frac{n(G)}{2}$. 
Hence, there are at least $\frac{n(G)}{2}$ vertices whose color appears at least twice on $G$ in $\alpha$ and $|\alpha(G)| \le |A| + \frac{(n(G) - |A|)}{2} \le \frac 34 n(G)$.
However, this bound is too weak for our purpose (proving that every vertex of $I_3$ has a free color). To improve it, we investigate the vertices whose colors appear twice.

\begin{lemma}\label{prop:no2into1}
    Let $v_1, v_2$ in $B$. If $v_1$ and $v_2$ have the same color $a$ in $\alpha$, then $a$ appears at least three times for $\alpha$. 
    Symmetrically, let $u_1$ and $u_2$ in $A$. If $u_1$ and $u_2$ have the same color $b$ in $\beta$, then $b$ appears at least three times for $\beta$. 
\end{lemma}
\begin{proof}
    Assume by contradiction that $v_1$ and $v_2$ are both colored with $a$ in $\alpha$ and that $a$ appears exactly twice for $\alpha$. 
    Note that $v_1$ and $v_2$ are in the same set $I_j$ since they have the same color in $\alpha$, so their respective cousins $\widetilde{v}_1$ and $\widetilde{v}_2$ are both adjacent to $v_1$ and $v_2$. 
    By Lemma~\ref{prop:matchingab}, if $a$ appears only once for $\beta$, it also appears only once for $\alpha$, which contradicts that both $v_1$ and $v_2$ are colored with $a$ in $\alpha$.
    Thus $a$ appears at least twice for $\beta$.
    
    Let $\gamma$ be the coloring obtained from $\alpha$ as follows (see Figure~\ref{fig:no2into1s} for an illustration).
    Recolor $\widetilde{v}_1$ with a color in $L(\widetilde{v}_1) \setminus (\alpha(N[\widetilde{v}_1]) \cup \{\beta(v_2) \})$, and $\widetilde{v}_2$ with a color in $L(\widetilde{v}_2) \setminus (\alpha(N[\widetilde{v}_2]) \cup \{\beta(v_1)\})$, which is possible since $v_1,v_2$ share the same color under $\alpha$ and are both adjacent to $\widetilde{v}_1$ and $\widetilde{v}_2$. 
    Then, recolor $v_1$ with $\beta(v_1)$, and $v_2$ with $\beta(v_2)$.
    Finally, recolor the vertices colored $a$ in $\beta$ with $a$.    
    Let $G'$ the graph obtained from $G$ by removing $v_1, v_2$ and $\beta^{-1}(a)$ and $L'$ the list-assignment obtained by restricting $L$ to $G'$. 
    By minimality of $G$ and Lemma~\ref{cl:mudecrease}, we have the following contradiction:
    \[
    \begin{aligned}
    \dist_G(\alpha, \beta) &\le 4 + |\beta^{-1}(a)| + n(G') + \mu(G') \\
     &\le 4 + |\beta^{-1}(a)| + (n(G) - 2 - |\beta^{-1}(a)|) + (\mu(G) - \min(|\beta^{-1}(a)|, |\{ v_1, v_2 \}|))  \\
     &\le n(G) + \mu(G).
    \end{aligned}
    \]
    By symmetry, we also get the result for $\beta$.
\end{proof}

\begin{figure}[ht]
        \centering
        \tikzstyle{v1}=[circle, minimum size=3pt, scale=0.6, fill]
        \tikzstyle{v2}=[circle,draw, minimum size=12pt, line width=2]
        \tikzstyle{fleche}=[arrows = {-Latex[width=4pt, length=4pt]}]
        \tikzstyle{labell}=[text opacity=1, scale =1]
        \tikzstyle{garrow}=[single arrow, draw, 
      minimum width =1mm, single arrow head extend=3pt, inner sep=0.5mm,
      minimum height=6mm, fill=\colora, \colora]
        \begin{tikzpicture}[scale=0.9]

        \nodc(a1)(1.5,0)[\colord, none];
        \nodc(a2)(0,0)[\colorg, \colord];
        \nodc(a3)(1.5,0.7)[\colorf, none];
        \nodc(a4)(0,0.7)[\colorg, \colorf];
        \nodc(a5)(-1.5,0)[none, \colorg];
        \nodc(a6)(-1.5,0.7)[none, \colorg];

        \draw[fleche] (a2) to (a1);
        \draw[fleche] (a4) to (a3);
        \draw[fleche] (a5) to (a2);
        \draw[fleche] (a6) to (a2);
        \draw[fleche] (a5) to (a4);
        \draw[fleche] (a6) to (a4);

        \node (h0) at (1.5,1.2) [labell]{$\widetilde{v}_1$};
        \node (h0) at (1.5, -0.5) [labell]{$\widetilde{v}_2$};
        \node (h0) at (0,1.2) [labell]{$v_1$};
        \node (h0) at (0, -0.5) [labell]{$v_2$};
        \node (h0) at (-1.6,1.25)[labell]{$\beta^{-1}(a)$};

        \draw[dashed, \colora] (2.3, -0.65) to  (2.3, 1.35);
        \node (f1) at (2.35,0.35) [garrow]{};

        \tikzset{xshift=4.7cm}

        \nodc(a1)(1.5,0)[\colorc, none];
        \nodc(a2)(0,0)[\colorg, \colord];
        \nodc(a3)(1.5,0.7)[\colorb, none];
        \nodc(a4)(0,0.7)[\colorg, \colorf];
        \nodc(a5)(-1.5,0)[none, \colorg];
        \nodc(a6)(-1.5,0.7)[none, \colorg];

        \draw[fleche] (a5) to (a2);
        \draw[fleche] (a6) to (a2);
        \draw[fleche] (a5) to (a4);
        \draw[fleche] (a6) to (a4);

        \draw[dashed, \colora] (2.3, -0.65) to  (2.3, 1.35);
        \node (f1) at (2.35,0.35) [garrow]{};

        \tikzset{xshift=4.7cm}

        \nodc(a1)(1.5,0)[\colorc, none];
        \nodc(a2)(0,0)[\colord, \colord];
        \nodc(a3)(1.5,0.7)[\colorb, none];
        \nodc(a4)(0,0.7)[\colorf, \colorf];
        \nodc(a5)(-1.5,0)[none, \colorg];
        \nodc(a6)(-1.5,0.7)[none, \colorg];

        \draw[dashed, \colora] (2.3, -0.65) to  (2.3, 1.35);
        \node (f1) at (2.35,0.35) [garrow]{};
        
        \tikzset{xshift=4.7cm}

        \nodc(a1)(1.5,0)[\colorc, none];
        \nodc(a2)(0,0)[\colord, \colord];
        \nodc(a3)(1.5,0.7)[\colorb, none];
        \nodc(a4)(0,0.7)[\colorf, \colorf];
        \nodc(a5)(-1.5,0)[\colorg, \colorg];
        \nodc(a6)(-1.5,0.7)[\colorg, \colorg];

        \end{tikzpicture}
        \caption{Reconfiguration sequence of Lemma~\ref{prop:no2into1} to obtain $\gamma$.
        The color violet appears at least twice in $\beta$, and exactly twice in $\alpha$: on $v_1$ and $v_2$.  We only represent some arcs of the color-shift digraph.}
        \label{fig:no2into1s}
\end{figure}
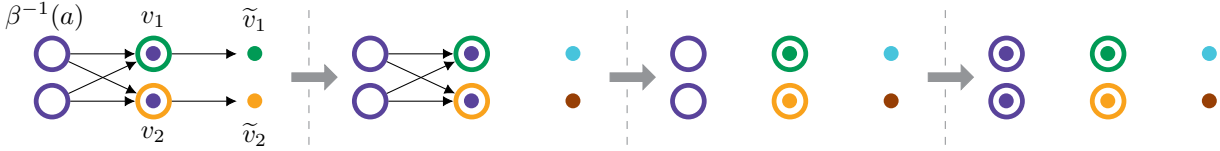 

To refine the bound on $|\alpha(G)|$, we prove that either more colors appear three times or more vertices have colors appearing twice for $\alpha$.

\begin{lemma}\label{prop:alpha23n}
The total number of colors in $\alpha$ and $\beta$ in $G$ is at most $\frac 23 n$.
\end{lemma}

\begin{proof}
For every $w \in A$, $\widetilde{w}$ is in $B$, so $\alpha(\widetilde{w})$ does not appear once for $\alpha$ by Lemma~\ref{prop:emptyab}.
Thus, we can separate the set $A$ into two sets $A_1 := \{w\in A\mid \alpha(\widetilde{w}) \text{ appears at least three times for } \alpha\}$ and $A_2 :=\{w\in A\mid \alpha(\widetilde{w}) \text{ appears twice for } \alpha\}$.

Let $T_i=\bigcup_{w\in A_i} \alpha^{-1}(\alpha(\widetilde{w}))$ for $i = 1,2$ be the set of vertices with the same color in $\alpha$ as the cousins of the vertices in $A_i$, see Figure~\ref{fig:alpha23n}.
By construction, the colors in $\alpha$ of the cousins of the vertices in $A_1$ and $A_2$ are distinct, so $T_1$ and $T_2$ are disjoint.
And since, for every $v \in T_1 \cup T_2$, at least two vertices are colored $\alpha(v)$, we have that $A$ and  $T_1 \cup T_2$ are disjoint. 
Hence, $A$, $T_1$ and $T_2$ are pairwise disjoint.

\begin{figure}[ht]
        \centering
        \tikzstyle{v1}=[circle, minimum size=3pt, scale=0.6, fill]
        \tikzstyle{v2}=[circle,draw, minimum size=12pt, line width=2]
        \tikzstyle{fleche}=[arrows = {-Latex[width=4pt, length=4pt]}]
        \tikzstyle{labell}=[text opacity=1, scale =1]
        \begin{subfigure}[t]{0.4\textwidth}
            \centering
            \begin{tikzpicture}[scale=1]
                \nodc(a1)(1.5,0)[\colorg, none];
                \nodc(a2)(0,0)[\colord, \colorg];
                \nodc(a3)(0,0.7)[\colord, none];
                \nodc(a4)(0,-0.7)[\colord, none];
                
                \node (h0) at (1.7, 0.5) [labell]{$w \in A_1$};
                \node (h0) at (-1, 0) [labell]{$T_1 \ni \widetilde{w}$};
                \node (h0) at (-0.7, 0.7) [labell]{$T_1 \ni$};
                \node (h0) at (-0.7, -0.7) [labell]{$T_1 \ni$};

                \draw[fleche] (a2) to (a1);
            \end{tikzpicture}
            
        \end{subfigure}
        \begin{subfigure}[t]{0.4\textwidth}
            \centering
            \begin{tikzpicture}[scale=1]
                
                \nodc(a1)(1.5,0)[\colorg, none];
                \nodc(a2)(0,0)[\colorf, \colorg];
                \nodc(a4)(0,0.7)[\colorf, none];
                
                \nodc(a5)(0,-0.7)[none, none];
        
                \draw[fleche] (a2) to (a1);

                \node (h0) at (1.5, 0.5) [labell]{$w \in A_2$};
                \node (h0) at (-1, 0) [labell]{$T_2 \ni \widetilde{w}$};
                \node (h0) at (-0.7, 0.7) [labell]{$T_2 \ni$};
            \end{tikzpicture}
        \end{subfigure}
        \caption{Let $w$ in $A$ colored violet in $\alpha$. 
        Its cousin $\widetilde{w}$ is the only vertex colored violet in $\beta$.
        On the left, the color of $\widetilde{w}$ in $\alpha$, orange here, appears at least three times in $\alpha$. 
        Hence, $w$ is in $A_1$ and the vertices colored orange in $\alpha$ are in $T_1$.
        On the right, the color of $\widetilde{w}$ in $\alpha$, green here, appears exactly twice. 
        Hence, $w$ is in $A_2$ and the two vertices colored green in $\alpha$ are in $T_2$. } 
        \label{fig:alpha23n}
\end{figure}
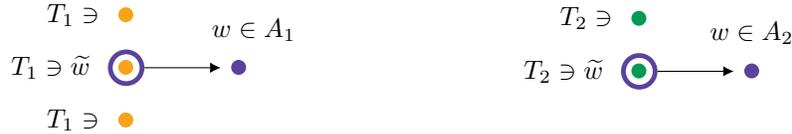 
Since every $\widetilde{w}$ is in $T_1$ for every $w \in A_1$, we have that $|T_1| \geq |A_1|$, and by construction of $A_1$, $\alpha^{-1}(\alpha(\widetilde{w}))$ contains at least three vertices, so $|\alpha(T_1)| \leq |T_1|/3$.
By Lemma~\ref{prop:no2into1}, for $w_1$ and $w_2$ in $A_2$, $\alpha^{-1}(\alpha(\widetilde{w}_1)) \neq \{ \widetilde{w}_1, \widetilde{w}_2 \}$.
Hence, the sets $\alpha^{-1}(\alpha(\widetilde{w}))$ for $w \in A_2$ (which have size $2$ by construction of $A_2$) are pairwise disjoint. In particular, $|\alpha(T_2)| = |A_2|$ and $|T_2| = 2|A_2|$ by construction of $T_2$.

Let $R := V(G) \setminus (A \sqcup T_1 \sqcup T_2)$. Since $A \cap R = \emptyset$, every color in $\alpha(R)$ appears at least twice in $R$, hence $|\alpha(R)| \leq |R|/2 $.
Now, we can bound $|\alpha(G)|$ (and $|\beta(G)|$ by symmetry):
    \[
    \pushQED{\qed} 
    \begin{aligned}
    |\alpha(G)|   &\leq |\alpha(A)| + |\alpha(T_1)| + |\alpha(T_2)| + |\alpha(R)| \\
                        &\leq |A| + \frac{|T_1|}{3} + |A_2| +  |R|/2 \\
                        &= |A| + \frac{|T_1|}{3} + |A_2| +  \frac{n - |A| - |T_1| - 2|A_2|}{2} \\
                        &= \frac{n}{2} + \frac{|A|}{2} - \frac{|T_1|}{6} \leq  \frac{n}{2} + \frac{3|A_2|+2|A_1|}{6} \\
                        &\leq  \frac{n}{2} + \frac{|A_2|+|T_2|+|A_1|+|T_1|}{6} \leq \frac{n}{2} + \frac{n}{6} = \frac 23 n.
    \end{aligned}
    \]
    
\end{proof}

We have bounded the number of colors in $\alpha(G)$ and $\beta(G)$.
Hence, we now use that there are few colors in the colorings to find free colors and recolor the graph with few steps. We proceed by recoloring $I_3$ and its neighborhoods in the color-shift digraph. To this end, we start by finding a free color for each vertex of $I_3$.

\begin{lemma}\label{prop:freeI3}
For each $v\in I_3$, $|L(v)| > |\alpha(G)| = |\beta(G)|$.
\end{lemma}
\begin{proof}
    First, observe that $|I_3| \le \frac{n}{3}$ (otherwise, each of $I_1, I_2$ and $I_3$ is larger than $\frac{n}{3}$). 
    By Lemma~\ref{prop:alpha23n}, $|L(v)| \ge d(v) + 2 = n - |I_3| + 2 \ge \frac{2}{3}n + 2 > |\alpha(G)|$.
\end{proof}

This implies that we can recolor the colorings to agree on $I_3$ using two steps for each vertex. We show it implies two properties on $I_3$ that will conclude the proof of Theorem~\ref{thm:rpartite}.
First, we show that given a color $a\in\alpha(I_3)$, there are more vertices colored with $a$ in $\alpha$ than in $\beta$. 
Then, we show that every color in $\alpha(I_3) \cup \beta(I_3)$ appears on at least three vertices of $I_3$.

\begin{lemma} \label{prop:colorsetI3}
    Let $a \in \alpha(I_3)$ and $b \in \beta(I_3)$. Then, $|\alpha^{-1}(a)| > |\beta^{-1}(a)|$ and $|\beta^{-1}(b)| > |\alpha^{-1}(b)|$.
\end{lemma}
\begin{proof}
    Assume by contradiction that $|\alpha^{-1}(a)| \leq |\beta^{-1}(a)|$ for a color $a \in \alpha(I_3)$.
    Note that $I_3$ contains all the vertices colored $a$ in $\alpha$.
    By Lemma~\ref{prop:freeI3}, let $\gamma$ be the coloring obtained from $\alpha$ by recoloring each vertex $v$ colored $a$ in $\alpha$ with a free color, then recoloring each vertex colored $a$ in $\beta$ with $a$ in $\alpha$.
    Finally, let $\zeta$ be the coloring obtained from $\beta$ by recoloring each vertex $v$ colored $a$ in $\alpha$ with $\gamma(v)$ (since these colors are free).
    
    Let $G'$ the graph obtained from $G$ by removing $\alpha^{-1}(a)$ and $\beta^{-1}(a)$, and $L'$ the list-assignment obtained by restricting $L$ to $G'$. 
    By minimality of $G$ and Lemma~\ref{cl:mudecrease}, we have the following contradiction: 
    \begin{align*}
        \dist_G(\alpha, \beta) &\le 2|\alpha^{-1}(a)| + |\beta^{-1}(a)| + d_{G'}(\gamma,\zeta) \\
                            &\le  2|\alpha^{-1}(a)| + |\beta^{-1}(a)| + n(G') + \mu(G') \\
                            &\le  2|\alpha^{-1}(a)| + |\beta^{-1}(a)| + (n(G) - |\alpha^{-1}(a)| - |\beta^{-1}(a)|) + (\mu(G) - \min(|\alpha^{-1}(a)|,|\beta^{-1}(a)|)) \\
                            &= n(G) + \mu(G) + |\alpha^{-1}(a)| - \min(|\alpha^{-1}(a)|,|\beta^{-1}(a)|)) \\
                            &\le n(G) + \mu(G).
    \end{align*}
    By symmetry of $\alpha$ and $\beta$, the conclusion follows.
\end{proof}

By Lemma~\ref{prop:colorsetI3} and Lemma~\ref{prop:matchingab}, we deduce that every color of $\alpha(I_3) \cup \beta(I_3)$ appears at least three times and hence we have that:
\begin{center}
    $\alpha(I_3)$ and $\beta(I_3)$ have both size at most $\frac{|I_3|}{3}$. 
\end{center}

Let us now conclude the proof of Theorem~\ref{thm:rpartite}. 
Let  $U:= \alpha^{-1}(\beta(I_3))$ be the in-neighborhood and $W := \beta^{-1}(\alpha(I_3))$ the out-neighborhood of $I_3$ in the color-shift digraph.
By symmetry, say $ |U| \le |W|$.
Let $\sigma$ be the coloring obtained from $\alpha$ as follows.
First, successively recolor each vertex $u$ of $U$ with an available color not in $\beta(I_3)$.
Note that this is always possible since at each step $I_3$ stays colored as in $\alpha$, hence at most $|N(u)\setminus I_3|+|\alpha(I_3)|+|\beta(I_3)|\leq \deg(u)-|I_3|/3$ colors are forbidden for $u$. 
Then, recolor the vertices $v$ of $I_3$ with $\beta(v)$, and finally the vertices $w$ of $W$ with $\beta(w)$ (see Figure~\ref{fig:endmulti}).

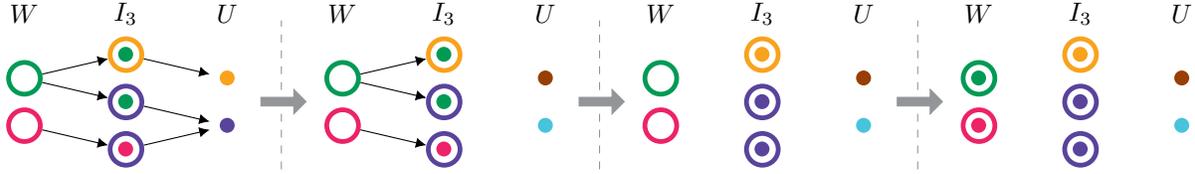
\begin{figure}[hbtp]
        \begin{center}
        \tikzstyle{v1}=[circle, minimum size=3pt, scale=0.6, fill]
        \tikzstyle{v2}=[circle,draw, minimum size=12pt, line width=2]
        \tikzstyle{fleche}=[arrows = {-Latex[width=4pt, length=4pt]}]
        \tikzstyle{labell}=[text opacity=1, scale =1]
        \tikzstyle{garrow}=[single arrow, draw, 
      minimum width =1mm, single arrow head extend=3pt, inner sep=0.5mm,
      minimum height=6mm, fill=\colora, \colora]
        \begin{tikzpicture}[scale=0.9]

        \nodc(a1)(0,-0.7)[\colore, \colorg];
        \nodc(a2)(0,0)[\colorf, \colorg];
        \nodc(a3)(0,0.7)[\colorf, \colord];
        \nodc(a4)(-1.5,0.35)[none, \colorf];
        \nodc(a5)(-1.5,-0.35)[none, \colore];
        \nodc(a6)(1.5,0.35)[\colord, none];
        \nodc(a7)(1.5,-0.35)[\colorg, none];

        \draw[fleche] (a5) to (a1);
        \draw[fleche] (a4) to (a2);
        \draw[fleche] (a4) to (a3);
        \draw[fleche] (a1) to (a7);
        \draw[fleche] (a2) to (a7);
        \draw[fleche] (a3) to (a6);

        \node (h0) at (-1.5,1.3) [labell]{$W$};
        \node (h0) at (1.5, 1.3) [labell]{$U$};
        \node (h0) at (0,1.3) [labell]{$I_3$};
        
        \draw[dashed, \colora] (2.3, -1) to  (2.3, 1.2);
        \node (f1) at (2.3,-0) [garrow, fill=\colora, \colora]{};
        \tikzset{xshift=4.7cm}

        \nodc(a1)(0,-0.7)[\colore, \colorg];
        \nodc(a2)(0,0)[\colorf, \colorg];
        \nodc(a3)(0,0.7)[\colorf, \colord];
        \nodc(a4)(-1.5,0.35)[none, \colorf];
        \nodc(a5)(-1.5,-0.35)[none, \colore];
        \nodc(a6)(1.5,0.35)[\colorc, none];
        \nodc(a7)(1.5,-0.35)[\colorb, none];

        \draw[fleche] (a4) to (a3);
        \draw[fleche] (a4) to (a2);
        \draw[fleche] (a5) to (a1);

        \node (h0) at (-1.5,1.3) [labell]{$W$};
        \node (h0) at (1.5, 1.3) [labell]{$U$};
        \node (h0) at (0,1.3) [labell]{$I_3$};

        \draw[dashed, \colora] (2.3, -1) to  (2.3, 1.2);
        \node (f1) at (2.3,-0) [garrow, fill=\colora, \colora]{};
        \tikzset{xshift=4.7cm}

        \nodc(a1)(0,-0.7)[\colorg, \colorg];
        \nodc(a2)(0,0)[\colorg, \colorg];
        \nodc(a3)(0,0.7)[\colord, \colord];
        \nodc(a4)(-1.5,0.35)[none, \colorf];
        \nodc(a5)(-1.5,-0.35)[none, \colore];
        \nodc(a6)(1.5,0.35)[\colorc, none];
        \nodc(a7)(1.5,-0.35)[\colorb, none];

        \node (h0) at (-1.5,1.3) [labell]{$W$};
        \node (h0) at (1.5, 1.3) [labell]{$U$};
        \node (h0) at (0,1.3) [labell]{$I_3$};

        \draw[dashed, \colora] (2.3, -1) to  (2.3, 1.2);
        \node (f1) at (2.3,-0) [garrow, fill=\colora, \colora]{};
        \tikzset{xshift=4.7cm}

        \nodc(a1)(0,-0.7)[\colorg, \colorg];
        \nodc(a2)(0,0)[\colorg, \colorg];
        \nodc(a3)(0,0.7)[\colord, \colord];
        \nodc(a4)(-1.5,0.35)[\colorf, \colorf];
        \nodc(a5)(-1.5,-0.35)[\colore, \colore];
        \nodc(a6)(1.5,0.35)[\colorc, none];
        \nodc(a7)(1.5,-0.35)[\colorb, none];

        \node (h0) at (-1.5,1.3) [labell]{$W$};
        \node (h0) at (1.5, 1.3) [labell]{$U$};
        \node (h0) at (0,1.3) [labell]{$I_3$};

        \end{tikzpicture}
        \end{center}
        \caption{Recoloring sequence to obtain $\sigma$. We represent a subset of $I_3$ and its neighborhoods in the color-shift digraph.} 
        \label{fig:endmulti}
\end{figure} 
Let $G' := G - I_3 - W$ and $L'$ obtained by restricting $L$ to $G'$.
By minimality of $G$, we can recolor $G'$ from $\sigma$ to $\beta$ in at most $n(G') + \mu(G')$ steps.
By Lemma~\ref{prop:colorsetI3}, we have that $|W| < |I_3|$ so $\mu(G')=\mu(G)-\min(|I_3|,|W|)=\mu(G)-|W|$ by Lemma~\ref{cl:mudecrease}. Since $|U| \le |W| < |I_3|$, we can recolor $G$ from $\alpha$ to $\beta$ in at most $|U| + |I_3| + |W| + (n(G) - |I_3| - |W|) + (\mu(G) - \min(|I_3|, |W|)) \le n(G) + \mu(G)$ steps, a contradiction.

\paragraph*{Acknowledgments} 
    The author would like to thank Nicolas Bousquet and Théo Pierron for their valuable feedback on this paper, and the authors of \citet*{cambie2024optimally} for their helpful remarks on the early version of the arXiv preprint.

\bibliographystyle{plainnat}

\end{document}